\documentclass[12pt]{article}
\usepackage{xparse}
\usepackage{amsmath}
\usepackage{amsfonts,amsthm,amssymb,bbding}
\usepackage{times}
\usepackage{enumerate}
\allowdisplaybreaks[4]
\usepackage{hyperref}
\hypersetup{
	colorlinks=true,
	anchorcolor=yellow,
	linkcolor=cyan,
	filecolor=blue,
	urlcolor=red,
	citecolor=green,
}
\usepackage[numbers,sort&compress,comma,square]{natbib}
\usepackage{latexsym,bm}
\usepackage{mathtools}

\evensidemargin=\oddsidemargin\topmargin=-1.5cm

\theoremstyle{definition}

\usepackage[margin=1in]{geometry}

\theoremstyle{definition}
\newtheorem{definition}{Definition}[section]
\theoremstyle{definition}

\theoremstyle{plain}

\newtheorem{theorem}{Theorem}
\newtheorem{lemma}{Lemma}
\newtheorem{corollary}{Corollary}
\newtheorem{proposition}{Proposition}

\theoremstyle{remark}
\newtheorem{remark}{Remark}
\newtheorem*{example}{Example}

\newtheoremstyle{case}{}{}{}{}{}{:}{ }{}
\theoremstyle{case}

\begin{document}
\title{\bf Zero product determined Banach algebras\footnote{Supported by the National Natural Science Foundation of China (Grant Nos. 11871021).}}
\author{ Jiankui Li \footnote{Corresponding author. Email: jkli@ecust.edu.cn }, Shaoze Pan, Shanshan Su \\[2mm]
\small School of Mathematics, East China University of Science and Technology, \\
\small  Shanghai 200237, P.R. China\\}

\date{}
\maketitle
{\flushleft\large\bf Abstract}
 Let $\mathcal{L}$ be a completely distributive commutative subspace lattice or a subspace lattice with two atoms, we use a unified approach to study the derivations, homomorphisms on $\mathrm{Alg} \mathcal{L}$. We verify that the multiplier algebra of $\mathrm{Alg} \mathcal{L}\cap \mathcal{K}(\mathcal{H})$  is isomorphic to $\mathrm{Alg} \mathcal{L}$ and $\mathrm{Alg} \mathcal{L}$ is zero product determined.  For $T$ in $M_{n}(\mathbb{C})$, $n\geq 2$, we show that $\mathcal{A}_{T}$ is zero product determined if and only if every local derivation from $\mathcal{A}_{T}$ into any Banach $\mathcal{A}_{T}$-bimodule is a derivation. In addition, we establish some equivalent conditions for an algebra to be zero product determined. For countable dimensional locally matrix algebras and triangular UHF algebras, we also show that they are zero Lie product determined.
\begin{flushleft}
\end{flushleft}

\section{Introduction}
Zero product determined algebras, introduced by Bre\v sar et al. in the two papers \cite{breasr2007,bilinearbresar}, has been intensively studied over the last decades  in which  it is used  as a powerful tool in studying some algebraic and analytic problems, see the  papers \cite{orthalam,commutators} and references therein. Especially, the zero product determined condition could be used to solve problems about derivations and homomorphisms, and it is also used to deal with the preserving problems, for relative details we recommend the book \cite{bresarbook}. There are some algebras which are zero product determined such as $C^{*}$-algebras and group algebras of locally compact groups. Recall that a zero product determined algebra would transform into a zero Lie product determined algebra when the product operation is replaced by Lie product. In this paper, we shall verify that some algebras are zero (Lie) product determined.

Suppose $\mathcal{H}$ is a complex separable Hilbert space, a \textit{subspace lattice} on $\mathcal{H}$ is a collection $\mathcal{L}$ of closed subspaces of $\mathcal{H}$ with $(0)$ and $\mathcal{H}$ in $\mathcal{L}$ and such that, for every family $\{S_{r}\}$ of elements of $\mathcal{L}$, both $\cap S_{r}$ and $\vee S_{r}$ belong to $\mathcal{L}$, where $\vee$ denotes the closed linear span of $\{S_{r}\}$. If $\mathcal{L}$ is a subspace lattice, Alg$\mathcal{L}$ denotes the set of operators on $\mathcal{H}$ which leave every subspace of $\mathcal{L}$ invariant. A subspace lattice $\mathcal{L}$ on $\mathcal{H}$ is said to have the \textit{strong rank one density property} if the algebra generated by the rank one operators of Alg$\mathcal{L}$ is dense in  Alg$\mathcal{L}$ in the strong operator topology. Notice that every completely distributive commutative subspace lattice and subspace lattice  with  two atoms has this property, for more details we refer to \cite{ABSL}. Nevertheless, as point out in \cite{ABSL}, the question about which subspace lattices $\mathcal{L}$ have the strong rank one density property is still unknown in general although there are some results to answer the question partially.

In this paper, multiplier algebras will also be concerned, which are used as a method to study whether a Banach algebra is zero product determined. The theory of multipliers is developed for topological algebras by Johnson and further investigated in the case of Banach algebras and $C^{*}$-algebras by Busby, Fontenot and others; see the monographs \cite{muitiplier} for additional references.

The paper is organized as follows:

In Section \ref{sec3}, we present some equivalent descriptions of zero product determined Banach algebras (Theorems \ref{equivalent} and \ref{multiplierB}). We investigate an algebra is zero product determined through characterizing its multipliers and local derivations, and  show  the graph algebras associated with derivations and homomorphisms are zero product determined.

In Section \ref{sec2}, we deal with the class of subspace lattice algebras $ \mathrm{Alg}\mathcal{L}$ where every subspace lattice $\mathcal{L}$ has  the strong rank one density property. Firstly, we prove that the multiplier algebra of Alg$\mathcal{L} \cap \mathcal{K(H)}$ and Alg$\mathcal{L}$ are isomorphic (Theorem \ref{strongone}). Secondly, we show that completely distributive commutative subspace lattice algebras and subspace lattice algebras with two atoms are necessarily zero product determined (Theorem \ref{cdcslB}).  As an application, we obtain every bounded local $n$-cocycle from $(\textrm{Alg}\mathcal{L}) ^{(n)}$ into a Banach Alg$\mathcal{L}$-bimodule $\mathcal{M}$ is an $n$-cocycle and every commutator in Alg$\mathcal{L}$ lies in the closed linear span of square-zero elements.

In Section \ref{zLpd}, we utilize a unified technique to prove some algebras are zero Lie product determined. Specifically, triangular UHF algebras, countable dimensional locally matrix algebras, and the algebra $\textrm{Alg} \mathcal{L}\cap \mathcal{F}(\mathcal{H})$ of a $\mathcal{J}$-subspace lattice algebra would be concerned in this section.


\section{Multiplier algebras and zero product determined Banach algebras}

Throughout this paper, $\mathcal{H}$ is a complex separable Hilbert space and $B(\mathcal{H})$ is the algebra of all bounded linear operators acting on $\mathcal{H}$. In this section, we shall prove some non-self-adjoint algebras are zero product determined.
\subsection{Zero product determined Banach algebras }\label{sec3}
Let us introduce some notations and concepts that we shall use in the following.  A Banach algebra $\mathcal{A}$ is said to be \textit{zero product determined}, if for every continuous bilinear functional $\varphi: \mathcal{A}\times \mathcal{A}\rightarrow \mathbb{C}$ satisfying $\varphi(x, y) = 0$ whenever $xy = 0$, there exists a continuous linear functional $\tau$ on $\mathcal{A}$ such that
\begin{equation}\label{eq2}
  \varphi(x, y)=\tau(xy)
\end{equation}
for any $x,y \in \mathcal{A}.$

If $\mathcal{A}$ is a Banach algebra with a bounded approximate identity, (\ref{eq2}) is  equivalent to
\begin{equation}\label{eq3}
  \varphi(xy,z)=\varphi(x,yz)
\end{equation}
for any $x,y,z \in \mathcal{A}.$

  A linear mapping $\delta$ from a Banach algebra $\mathcal{A}$ to a Banach $\mathcal{A}$-bimodule $\mathcal{M}$ is said to be a \textit{derivation} if
$\delta(xy) = \delta(x) \cdot y + x \cdot \delta(y)$ for any $x,y\in \mathcal{A}.$
 A linear mapping $\delta$ from a Banach algebra $\mathcal{A}$ to a Banach $\mathcal{A}$-bimodule $\mathcal{M}$ is called a \textit{local derivation} if, for every $x\in \mathcal{A}$, there exists a derivation $\delta_{x} : \mathcal{A}\rightarrow \mathcal{M}$
such that $\delta(x) = \delta_{x}(x).$

Let $\mathcal{M}$ be a Banach $\mathcal{A}$-bimodule over a Banach algebra $\mathcal{A}$. A linear mapping $V: \mathcal{A}\rightarrow \mathcal{M}$ is said to be a \textit{left (right) multiplier} if $V(ab)=V(a)\cdot b$ ($V(ab)=a\cdot V(b)$), for every $a,b\in \mathcal{A}.$ A linear mapping $V: \mathcal{A}\rightarrow \mathcal{M}$ is called a \textit{local left (right) multiplier} if, for every $a\in\mathcal{A}$, there exists a left (right) multiplier $V_{a}$ from $\mathcal{A}$ to $\mathcal{M}$ such that $V(a)=V_{a}(a).$

%
%
%
%
%
%

For  $T\in B(\mathcal{H})$, we denote  $\mathcal{A}_{T}$ the norm-closed subalgebra of $B(\mathcal{H})$ generated by $T$ and $I$. If $\textrm{dim}\mathcal{H}<\infty$,  we obtain that $\mathcal{A}_{T}$ is zero product determined if every bounded local derivation from $\mathcal{A}_{T}$ into a  Banach $\mathcal{A}_{T}$-bimodule $\mathcal{M}$ is a derivation. In order to prove this, some significant examples are needed. The first example shows that there exists a local derivation from $\mathcal{J}_{n}, n\geq2,$ into some of its bimodule that is not a derivation.

\begin{example} \label{Fanli1}
For a positive number $n \geq 2$,  let $\{ E_{ij} \}_{i,j = 1}^{n}$ be the standard matrix unit system of $M_{n}(\mathbb{C})$, and let
$\mathcal{J}_{n}$ be the algebra generated by the identity $I$ and a Jordan block $D_{n}$, where
$$D_{n}=\begin{pmatrix}
0 & 1 & 0&\cdots&0\\
0 & 0 & 1 & \cdots & 0\\
\vdots & \vdots &\vdots&\ddots & \vdots\\
0 & 0 & 0 & \cdots & 1\\
0 & 0 & 0&\cdots & 0
\end{pmatrix}_{n \geq 2}. $$

For any element $A\in \mathcal{J}_{n}$, it can be written as $A=\lambda_{1}I+\lambda_{2}D_{n}+\cdots+\lambda_{n}D_{n}^{n-1},$ where $ \lambda_{i} \in \mathbb{C}.$

%

For $n=2$, $$\mathcal{J}_{2}=\left\{\begin{pmatrix}
               \lambda_{1} & \lambda_{2} \\
                0 & \lambda_{1} \\
             \end{pmatrix}: \lambda_{1}, \lambda_{2}\in \mathbb{C}\right\}=\{\lambda_{1}I+\lambda_{2}E_{12}: \lambda_{1}, \lambda_{2}\in \mathbb{C}\}.$$
Let $\mathcal{M}$ be a \textit{non-unital} Banach $\mathcal{J}_{2}$-bimodule with its left module action defined by $A\cdot M=0$, for every $A\in \mathcal{J}_{2}, M \in \mathcal{M},$ and it satisfies $\delta(I)\cdot E_{12} \neq 0$. Define
$$\delta: \mathcal{J}_{2}\rightarrow \mathcal{M}, \ \lambda_{1}I+\lambda_{2}E_{12}\mapsto \lambda_{1}\delta(I)\cdot I+2\lambda_{2}\delta(I)\cdot E_{12}.$$
Then $\delta$ is a local derivation which is not a derivation.

Indeed, if $\lambda_{1}\neq 0,$ then $\delta(\lambda_{1}I+\lambda_{2}E_{12})=(\delta(I)\cdot I+\frac{\lambda_{2}}{\lambda_{1}}\delta(I)\cdot E_{12})(\lambda_{1}I+\lambda_{2}E_{12})$ or otherwise $\delta(\lambda_{2}E_{12})=2\delta(I)\cdot \lambda_{2}E_{12}.$ This shows that $\delta$ is a local derivation. As $\delta((I + E_{12})^{2})-\delta(I + E_{12}) \cdot (I + E_{12}) = \delta(I) \cdot E_{12} \neq 0$, $\delta$ is not a derivation.


For $n\geq3$, define a linear mapping $\delta : \mathcal{J}_{n} \rightarrow \mathcal{J}_{n} , \ A \mapsto (\lambda_{3} - \lambda_{2})E_{1n}$. We  claim that $\delta$ is a local derivation which is not a derivation.

Define
$$\delta_{1} : \mathcal{J}_{n} \rightarrow \mathcal{J}_{n}, \ A \mapsto [E_{1,n-2}, A] $$
and
$$\delta_{2} : \mathcal{J}_{n} \rightarrow \mathcal{J}_{n}, \ A \mapsto [E_{2n}, A] .$$ It is evident that those mappings
are derivations. Moreover, for any $A \in \mathcal{J}_{n}$, we have
\begin{equation*}
\delta(A) = \begin{cases}
\delta_{1}(A), & \text{if} \quad \lambda_{2} = 0;\\
\lambda_{2}^{-1} (\lambda_{2} - \lambda_{3}) \delta_{2}(A), & \text{if} \quad \lambda_{2} \neq 0.
\end{cases}
\end{equation*}
Hence, $\delta$ is a local derivation.

However, $\delta$ is not a derivation, since with a straightforward calculation, the $(1, n)$-entry of
$$\delta(D_{n}^{2}) - \delta(D_{n})D_{n} - D_{n} \delta(D_{n})$$
is equal to $1$.

\end{example}
\begin{remark}
It is easy to verify that, in Example \ref{Fanli1}, every local derivation from $\mathcal{J}_{2}$ into any \textit{unital} Banach $\mathcal{J}_{2}$-bimodule is a derivation.
\end{remark}

In the following example, we construct a mapping  on $\mathcal{J}_{n}$ such that it is a local left (right) multiplier but not a left (right) multiplier.

\begin{example}\label{multiplier}
$\mathcal{J}_{n}$ is defined in the same way as Example \ref{Fanli1}. For any element  $A\in \mathcal{J}_{n}$, $A=\lambda_{1}I+\lambda_{2}D_{n}+\cdots+\lambda_{n}D_{n}^{n-1},$ where $ \lambda_{i} \in \mathbb{C}.$

Define a linear mapping $\delta: \mathcal{J}_{n}\rightarrow \mathcal{J}_{n}$ by $A\mapsto \lambda_{1}D_{n}^{n-2}+2\lambda_{2}D_{n}^{n-1}.$

If $\lambda_{1}\neq 0$, $$\delta(A)=(D_{n}^{n-2}+\frac{\lambda_{2}}{\lambda_{1}}
D_{n}^{n-1})(\lambda_{1}I+\lambda_{2}D_{n}+\cdots+\lambda_{n}D_{n}^{n-1}).$$

If $\lambda_{1}=0$, $\delta(A)=2\lambda_{2}D_{n}^{n-1}.$

This shows that $\delta$ is a local multiplier. If $\lambda_{1}\neq 0$,  $D_{n}^{n-2}+\frac{\lambda_{2}}{\lambda_{1}}
D_{n}^{n-1}$ is uniquely determined by $A$, and so $\delta$ is not  a multiplier.
\end{example}

\begin{theorem}\label{equivalent}
For $T \in M_{n}(\mathbb{C}), n\geq 2$,  the following statements are equivalent:

\begin{itemize}
  \item[(i)] $\mathcal{A}_{T}$ is a zero product determined Banach algebra.
   \item[(ii)] Every bounded local derivation from $\mathcal{A}_{T}$ into any  Banach $\mathcal{A}_{T}$-bimodule $\mathcal{M}$ is a derivation.
   \item[(iii)]    Every bounded local left (right) multiplier  from $\mathcal{A}_{T}$ into any   Banach $\mathcal{A}_{T}$-bimodule $\mathcal{M}$ is a left (right) multiplier.
    \item[(iv)]  $T$ is diagonalizable.
\end{itemize}

\end{theorem}

\begin{proof}

$(i)\Rightarrow (iv)$: Denote the spectrum of $T$ by $\textrm{Sp}(T)=\{\lambda_{1}, \lambda_{2}, \cdots, \lambda_{s}\},$ where $\lambda_{i}\in \mathbb{C}, i=1,2,\cdots,s.$ By Spectral decomposition theorem,  $T$ can be written in the following form:
$$T=T_{1}\oplus T_{2}\oplus\cdots\oplus T_{s},$$
 where $T_{i}$ is the operator corresponding to the eigenvalue $\lambda_{i}.$  By \cite[Theorem 2.1]{wupeiyuan},

 $$\mathcal{A}_{T}=\mathcal{A}_{T_{1}}\oplus\mathcal{A}_{T_{2}}\oplus\cdots\oplus\mathcal{A}_{T_{s}}.$$

It follows from  \cite[Theorem 1.16]{bresarbook} that $\mathcal{A}_{T}$ is zero product determined if and only if $\mathcal{A}_{T_{i}}$ is zero product determined for every $i=1,2, \ldots, s.$ For each $\mathcal{A}_{T_{i}}, \ i\in\{1,\ldots, s\}$, $T_{i}$ is similar to a matrix which can be written as following:
$$\left(
  \begin{array}{ccccc}
    J_{k_{1}} & 0 & 0 & \cdots & 0 \\
    0 & J_{k_{2}}& 0 &\cdots & 0\\
     \vdots &\vdots& \vdots &\ddots & \vdots\\
    0 & 0 & 0 & \cdots & 0\\
    0 & 0 & 0 & \cdots & J_{k_{m}}\\
  \end{array}
\right)$$
where $J_{k_{j}}$ is the Jordan block that has the form:
$$J_{k_{j}}(\lambda_{i})=\left(
          \begin{array}{ccccc}
            \lambda_{i} & 1 & 0 &\cdots& 0\\
            0 & \lambda_{i} & 1 &\cdots& 0\\
            \vdots & \vdots & \ddots& \ddots & \vdots\\
            0 &  0&  \cdots& \lambda_{i}& 1\\
            0 & 0 & \ldots &0& \lambda_{i} \\
          \end{array}
        \right)_{k_{j}\times k_{j}}.
$$
For each $J_{k_{j}}$,
$$\lambda_{i}I-J_{k_{j}}(\lambda_{i})=\left(
          \begin{array}{ccccc}
            0 & 1 & 0 &\cdots& 0\\
            0 & 0 & 1 &\cdots& 0\\
            \vdots & \vdots & \ddots& \ddots & \vdots\\
            0 &  0&  \cdots& 0& 1\\
            0 & 0 & \ldots &0& 0 \\
          \end{array}
        \right)_{k_{j}\times k_{j}},
$$
which is denoted by ${\widetilde{J}_{k_{j}}}.$

Let $k=\max\{k_{1},k_{2},\ldots, k_{m}\}$. If $k\geq2$, denote
$\mathcal{B}_{{{\widetilde{J}_{{k}}}}}$  the non-unital algebra generated by ${\widetilde{J}_{{k}}}$. Apparently, $\mathcal{B}_{{\widetilde{J}_{{k}}}}^{2}\neq \mathcal{B}_{{\widetilde{J}_{{k}}}}.$ By \cite[Proposition 1.19]{bresarbook}, $\mathcal{A}_{{\widetilde{J}_{{k}}}}$,  the algebra obtained by adjoining a unity to $\mathcal{B}_{{\widetilde{J}_{{k}}}},$ is not zero product determined. This contradicts to our assumption. Hence for each $J_{k_{j}}$, it is a diagonal matrix, that is,  $k_{j}$ must be 1. It follows that $T$ is  diagonalizable.


$(i)\Rightarrow(ii)$: If $\mathcal{A}_{T}$ is zero product determined,  then by \cite[Theorem 2.5(ii)]{ncocycles}, every bounded local derivation from $\mathcal{A}_{T}$ into any Banach $\mathcal{A}_{T}$-bimodule $\mathcal{M}$ is a derivation

$(ii)\Rightarrow (iii)$: In this case, every bounded local left (resp. right) multiplier can be considered as a local derivation according to define the module action by $\mathcal{J}_{n}\cdot \mathcal{M} =0$ (resp. $\mathcal{M}\cdot \mathcal{J}_{n}=0$).

$(iii)\Rightarrow (iv)$: Thanks to Example \ref{multiplier}, for $n \geq 2$, we could always construct a bounded local multiplier on $\mathcal{J}_{n}$ that is not a multiplier. Thus if $(iii)$ is satisfied, the dimension of each $\mathcal{A}_{T_{i}}$ must be 1, and therefore $T$ is diagonalizable.

$(iv)\Rightarrow(i)$: If $(iv)$ is satisfied, then by \cite[Corollary 2.11]{marcouxabelian}, $\mathcal{A}_{T}$ is similar to a $C^*$-algebra which is zero product determined, so is $\mathcal{A}_{T}$.
\end{proof}

Let $\mathcal{A}$ be a Banach algebra, $\mathcal{M}$  a Banach $\mathcal{A}$-bimodule. We say that $\mathcal{M}$
 is a \textit{separating} Banach $\mathcal{A}$-bimodule, if, for any $m\in \mathcal{M}$, $$\mathcal{A}\cdot m=\{0\}\Rightarrow m=0\ \textrm{and} \ m\cdot \mathcal{A}=\{0\}\Rightarrow m=0.$$

In the following theorem, the proof of $(ii)$ implies $(iii)$ could be seen in \cite{chen}, however, we provide the details here for completeness and convenience.

\begin{theorem}\label{multiplierB}
 Suppose $\mathcal{A}$ is a  Banach algebra having a bounded approximate identity. Then the following properties are equivalent:
 \begin{itemize}
   \item[(i)] If any continuous linear mapping $f$ from $\mathcal{A}$ into $\mathcal{A}^{*}$ satisfies the following property:  $$ab=0\Rightarrow f(a)\cdot b=0,$$ then $f$ is a left multiplier.
   \item[(ii)]  $\mathcal{A}$ is zero product determined.
   \item[(iii)] Suppose $\mathcal{M}$ is a separating Banach $\mathcal{A}$-bimodule, if any continuous linear mapping $f$ from $\mathcal{A}$ into $\mathcal{M}$ satisfies the following property:  $$ab=0\Rightarrow f(a)\cdot b=0,$$ then $f$ is a left multiplier.
 \end{itemize}
\end{theorem}
\begin{proof}
  $(i)\Rightarrow (ii)$:   For any  continuous bilinear mapping $\varphi: \mathcal{A}\times \mathcal{A}\rightarrow \mathbb{C}$ with the property that $\varphi(x, y) = 0$ whenever $xy = 0$, let $f$ be a continuous linear mapping from $\mathcal{A}$ into $\mathcal{A}^{*}$ by $$f(a)(c)=\varphi(a, c)$$ for any $a,c\in \mathcal{A}.$ Then, if $ab=0$, for any $c\in \mathcal{A}$, we have $$(f(a)\cdot b)(c)=f(a) (bc)=\varphi(a, bc)=0.$$
  It follows that $f(a)\cdot b=0$. According to the assumption, $f$ is a left multiplier, i.e. $f(ab)=f(a)\cdot b$ for any $a,b\in \mathcal{A}.$

   By the definition of $f$, for any $a,b,c \in \mathcal{A}$, we have $$\varphi(a, bc)=(f(a)\cdot b)(c)=f(ab)(c)=\varphi(ab, c).$$ Hence, $\mathcal{A}$ is zero product determined.

   $(ii)\Rightarrow(iii)$:  Let $f$ be any continuous linear mapping from $\mathcal{A}$ into $\mathcal{M}$ satisfying $$ab=0\Rightarrow f(a)\cdot b=0,$$ for any $a,b\in \mathcal{A}$. Define a continuous bilinear mapping $\phi$ from $\mathcal{A}\times \mathcal{A}$ into $\mathcal{M}$ by

     $$\phi: \mathcal{A}\times \mathcal{A} \rightarrow  \mathcal{M}, (a, b) \mapsto f(ab)-f(a)\cdot b. $$

   If $ab=0$, then $\phi(a,b)=0.$

    Since $\mathcal{A}$ is zero product determined, $$\phi(ab,c)=\phi(a,bc),$$ that is, $$f(abc)-f(ab)\cdot c=f(abc)-f(a)\cdot bc$$ for any $a,b,c\in \mathcal{A}$. Hence, $$f(ab)\cdot c=f(a)\cdot bc,$$ that is,  $$(f(ab)-f(a)\cdot b)\cdot c=0$$ for any $a,b,c\in \mathcal{A}.$ Since $\mathcal{M}$ is a separating Banach $\mathcal{A}$-bimodule, it follows that $f(ab)=f(a)\cdot b$ for any $a,b\in \mathcal{A}.$

    $(iii)\Rightarrow (i)$: This is evident.
\end{proof}
\begin{remark}
  Theorem \ref{multiplierB} can be used to weaken the conditions of some conclusions. For example, in \cite[Proposition 3.2]{ncocycles2008}, it requires a unital Banach algebra $\mathcal{A}$ satisfies the following two conditions:

  \begin{itemize}
    \item [(i)] For every unital Banach $\mathcal{A}$-bimodule $\mathcal{X}$, a bounded operator $D: \mathcal{A}\rightarrow \mathcal{X}$ is a left multiplier if and only if $ab= 0$ implies $D(a)b = 0.$
    \item [(ii)] For every unital Banach $\mathcal{A}$-bimodule $\mathcal{X}$, a bounded operator $D: \mathcal{A}\rightarrow \mathcal{X}$ satisfies: for $a_{0}, a_{1}, a_{2}\in \mathcal{A}$, $a_{0}a_{1}=a_{1}a_{2}=0\Rightarrow a_{0}D(a_{1})a_{2}=0$ if and only if
$D(acb)-aD(cb)-D(ac)b + aD(c)b = 0$
for all $a, b, c \in \mathcal{A}.$
  \end{itemize}
  According to Theorem \ref{multiplierB}, if the condition $(i)$ is satisfied, then $\mathcal{A}$ is zero product determined, which means that condition $(ii)$ holds. Hence, in \cite[Proposition 3.2]{ncocycles2008}, the conclusion holds as long as condition $(i)$ is satisfied.
\end{remark}

\begin{proposition}
  Let $\mathcal{A}$ be a unital Banach algebra. For any $a,b\in \mathcal{A}$, define $a\circ b$ by $$a\circ b=\phi(a)b$$ where $\phi$ is a character of $\mathcal{A}$. Then $\mathcal{A}$ is a zero product determined Banach algebra.
\end{proposition}

\begin{proof}
  Let $\varphi$ be any continuous bilinear mapping from $\mathcal{A}\times \mathcal{A}$ into $\mathbb{C}$ satisfying $\varphi(a, b)=0$ whenever $a\circ b=0.$

  For any $a,b\in \mathcal{A}$, we obtain  that $$\varphi(a-\phi(a)1, b)=0,$$ that is,
  \begin{equation}\label{7}
    \varphi(a,b)=\phi(a)\varphi(1,b).
  \end{equation}
  Thus, by (\ref{7}),
    \begin{eqnarray*}
    \varphi(a\circ b, c) &{=}& \phi(a\circ b)\varphi(1,c)\\
      &=& \phi(a)\phi(b)\varphi(1,c) \\
     &=& \phi(a)\varphi(1, \phi(b)c) \\
    &=& \phi(a)\varphi(1, b\circ c) \\
    &=& \varphi(a, b\circ c),
    \end{eqnarray*}
which completes the proof.
\end{proof}

If  $\mathcal{A}$ and $\mathcal{B}$ are Banach algebras,  and $\varphi$ is a bounded homomorphism from $\mathcal{A}$ into $\mathcal{B}$, then the graph of $\varphi$, $\mathcal{G}=\{(a, \varphi(a)): a\in \mathcal{A}\}$,  is a Banach algebra  with the following norm and operations: $$\|(a,\varphi(a))\|_{1}=\|a\|_{\mathcal{A}}+\|\varphi(a)\|_{\mathcal{B}},$$$$(a, \varphi(a))+(b, \varphi(b))=(a+b, \varphi(a)+\varphi(b))=(a+b, \varphi(a+b)),$$ $$(a, \varphi(a))(b, \varphi(b))=(ab, \varphi(a)\varphi(b))=(ab, \varphi(ab)),$$  for any $a,b \in \mathcal{A}$.

\begin{proposition}\label{graph}
If $\mathcal{A}$ and $\mathcal{B}$ are Banach algebras, $\varphi$ is a bounded homomorphism from $\mathcal{A}$ into $\mathcal{B}$ and $\mathcal{A}$ is zero product determined with a bounded approximate identity, then $\mathcal{G}=\{(a, \varphi(a)): a\in \mathcal{A}\}$ is a zero product determined Banach algebra.
\end{proposition}

\begin{proof}
We note that if $\{e_{i}\}$ is a bounded approximate identity of $\mathcal{A}$, then it is easy to check that $\{(e_{i}, \varphi(e_{i}))\}$ is a bounded approximate identity of $\mathcal{G}$.

  Let $\phi$ be a continuous bilinear mapping from $\mathcal{G}\times \mathcal{G}$ into $\mathbb{C}$ satisfying $\phi((a, \varphi(a)), (b, \varphi(b)))=0$ whenever $(a, \varphi(a))(b, \varphi(b))=0.$

  It is observed that  $$(a, \varphi(a))(b, \varphi(b))=0\Leftrightarrow ab=0,$$
   then we can define a mapping $\Phi$ from $\mathcal{A}\times \mathcal{A}$ into $\mathbb{C}$ by $$\Phi(a,b)=\phi((a, \varphi(a)), (b, \varphi(b))),$$ for any $a,b\in \mathcal{A}.$

  It is easy to see that $\Phi$ is a bounded bilinear mapping, and $\Phi(a,b)=0$ when $ab=0$ for any $a,b \in \mathcal{A}$.

  Since $\mathcal{A}$ is zero product determined, for any $a,b,c \in \mathcal{A}$, $$\Phi(ab,c)=\Phi(a,bc),$$
  that is, $$\phi((ab, \varphi(ab)), (c, \varphi(c)))=\phi((a, \varphi(a)),(bc, \varphi(bc))).$$
  Expanding the above equation, we obtain
  \begin{eqnarray*}
    \phi((a, \varphi(a))(b, \varphi(b)), (c, \varphi(c))) &=& \phi((ab, \varphi(ab)), (c, \varphi(c))) \\
     &=& \phi((a, \varphi(a)),(bc, \varphi(bc))) \\
    &=& \phi((a, \varphi(a)),(b, \varphi(b))(c, \varphi(c))).
  \end{eqnarray*}
  Hence,  $\mathcal{G}$ is zero product determined.
\end{proof}

Suppose $\mathcal{A}$ is a Banach algebra,  $\mathcal{M}$ is a Banach $\mathcal{A}$-bimodule, and  $\delta: \mathcal{A}\rightarrow \mathcal{M}$ is a bounded derivation. Then the graph of $\delta$, $\mathcal{G}=\{(a, \delta(a)): a\in \mathcal{A}\}$, is a Banach  algebra with the following norm and operations: $$\|(a, \delta(a))\|_{1}=\|a\|_{\mathcal{A}}+\|\delta(a)\|_{\mathcal{M}},$$  $$(a, \delta(a))+(b, \delta(b))=(a+b, \delta(a)+\delta(b))=(a+b, \delta(a+b)),$$ $$(a, \delta(a))(b, \delta(b))=(ab, \delta(a)b+a\delta(b))=(ab, \delta(ab)),$$ for any $a,b\in \mathcal{A}.$

 Similar process as Proposition \ref{graph}, we obtain the following result without giving the proof.
\begin{proposition}\label{0000}
  If $\mathcal{A}$ is a  Banach algebra which is zero product determined and has a bounded approximate identity, $\mathcal{M}$ is a Banach $\mathcal{A}$-bimodule,  and $\delta$ is a bounded derivation from $\mathcal{A}$ into $\mathcal{M}$, then $\mathcal{G}=\{(a, \delta(a)): a\in \mathcal{A}\}$ is a zero product determined Banach algebra.
\end{proposition}

\begin{proposition}
  Suppose $\mathcal{H}$ is a complex Hilbert space, $\mathcal{A}$ is a  Banach subalgebra of  $B(\mathcal{H})$ which is zero product determined and  has a bounded left approximate identity. If $P, Q\in \mathrm{Lat} \mathcal{A}$, where $\mathrm{Lat} \mathcal{A}$ denotes the set of invariant subspaces of $\mathcal{A}$, then $\mathcal{B}=\overline{(P-Q)\mathcal{A}(P-Q)}^{\|\cdot\|}$
  is a zero product determined Banach algebra.
\end{proposition}

\begin{proof}
   Define $$\Phi: \mathcal{A}\rightarrow (P-Q)\mathcal{A}(P-Q)$$ by $$A\mapsto (P-Q)A(P-Q),$$ for any $A\in \mathcal{A}$.
    Then $\phi$ is a bounded homomorphism, according to \cite[Corollary 5.7]{bresarbook}, $\mathcal{B}$ is zero product determined.
\end{proof}

The next purpose is to prove some non-self-adjoint algebras are zero product determined.

 A subspace lattice $\mathcal{L}$ on $\mathcal{H}$ is called  \textit{completely distributive commutative subspace lattice} if it is completely distributive and commutative.  A set $\mathcal{L}=\{0, M, N, \mathcal{H}\}$ is called \textit{subspace lattice with two atoms}  if $M\wedge N=\{0\}$ and $M\vee N=\mathcal{H}.$

It is known that every nest and complete atomic Boolean lattice is a completely distributive commutative subspace lattice.

\begin{theorem}\label{cdcslB}
  If $\mathcal{L}$ is a completely distributive commutative subspace lattice or a subspace lattice with two atoms, then $\mathrm{Alg}\mathcal{L}$ is zero product determined.
\end{theorem}

\begin{remark}
Theorem \ref{cdcslB} can be used to deduced  some significant consequences. In \cite{commutators}, the authors discuss the connection between commutators and square-zero elements in Banach algebras. Recall that, for an algebra $\mathcal{A}$, a \textit{commutator} in $\mathcal{A}$ is an element of the form $[x, y]:=xy-yx$ and a \textit{square-zero element} in $\mathcal{A}$ is an element $x$ such that $x^{2} = 0.$  Let $\mathcal{L}$ be as in Theorem \ref{cdcslB}, by \cite[Theorem 2.1]{commutators}, we obtain every commutator in Alg$\mathcal{L}$ lies in the closed linear span of square-zero elements.
\end{remark}
The proof of the main theorem will be given in the next subsection. To prove Theorem \ref{cdcslB}, we need  some conclusions about multiplier algebras that would be studied in the following.

\subsection{Multiplier algebras}\label{sec2}

 Let $\mathcal{A}$ be a complex Hausdorff topological algebra in which multiplication is associative and separately continuous. We call an algebra $\mathcal{A}$ is \textit{faithful} if, for any  $a\in \mathcal{A}$, $a\mathcal{A}=\mathcal{A}a=\{0\}$ implies $a=0.$

  Let $\mathcal{K}(\mathcal{H})$ be the set of compact operators in $B(\mathcal{H})$. If $\mathcal{A}$ is a subalgebra of $B(\mathcal{H})$, then $\mathcal{A} \cap \mathcal{K}(\mathcal{H})$ is the set of  all compact operators belonging to $\mathcal{A}$.


A \textit{multiplier} on $\mathcal{A}$ is a pair $(L, R): \mathcal{A}\rightarrow \mathcal{A}$ where $L$ and $R$ are linear mappings such that, for all $a,b\in \mathcal{A}$,
$$L(ab)=L(a)b, R(ab)=aR(b), \textrm{and}\ aL(b)=R(a)b.$$
Let $\mathcal{M}(\mathcal{A})$ be the unital algebra of all multipliers on $\mathcal{A}$, and it is called the \textit{multiplier algebra} of $\mathcal{A}.$ If $\mathcal{A}$ is faithful, then $\mathcal{A}$ can be considered as a subalgebra of $\mathcal{M}(\mathcal{A})$ by the following action:
$$\mathcal{A}\rightarrow \mathcal{M}(\mathcal{A}), a\mapsto (L_{a}, R_{a}),$$ where $L_{a}(b)=ab, R_{a}(b)=ba$ for all $a,b\in \mathcal{A}.$ This mapping is an isomorphism from $\mathcal{A}$ into $\mathcal{M}(\mathcal{A}).$ For the convenience of the reader we refer the relevant material \cite{multiplierstopo, MT,murphy}.


Let $\mathcal{A}$ be a faithful Banach algebra and $\mathcal{M}(\mathcal{A})$  the multiplier algebra of $\mathcal{A}$, then the \textit{strict topology} $s$ for $\mathcal{M}(\mathcal{A})$ is defined to be that locally convex topology generated by the seminorms $(\lambda_{a})_{a\in \mathcal{A}}$ and $(\rho_{a})_{a\in \mathcal{A}}$, where
$\lambda_{a}(x)=\|ax\|$ and $\rho_{a}(x)=\|xa\|$ for any $x\in \mathcal{M}(\mathcal{A})$.

%

Recall that if $\mathcal{A}$ is complete and has two-sided approximate identity (not necessarily bounded), then $\mathcal{A}$ is dense in $\mathcal{M}(\mathcal{A})$ under the strict topology (\cite[Theorem 3.2]{multiplierstopo}).






\begin{theorem}\label{strongone}
  If $\mathcal{L}$ is a subspace lattice with the strong rank one density property on a complex separable Hilbert space $\mathcal{H}$. Then $\mathcal{M}(\mathrm{Alg} \mathcal{L}\cap \mathcal{K}(\mathcal{H}))$ and $\mathrm{Alg}\mathcal{L}$ are isomorphic.
\end{theorem}

\begin{proof}
 We define a mapping $\varphi$ from $\mathrm{Alg} \mathcal{L}$ to $\mathcal{M}(\mathrm{Alg} \mathcal{L}\cap \mathcal{K}(\mathcal{H}))$ by $$a\mapsto (L_{a}, R_{a}).$$ Obviously,  this definition is well defined, and it is easy to check that $\varphi$ is a homomorphism.

 We claim this mapping is faithful. Suppose $(L_{a}, R_{a})=0$, then $L_{a}b=0$ for any $b\in \mathrm{Alg} \mathcal{L}\cap \mathcal{K}(\mathcal{H}).$ Observe that $(\mathrm{Alg} \mathcal{L}\cap \mathcal{K}(\mathcal{H}))^{**}=\mathrm{Alg} \mathcal{L}$, by Goldstine’s Theorem, the unit ball of $\mathrm{Alg} \mathcal{L}\cap \mathcal{K}(\mathcal{H})$ is $\sigma$-weak dense in the unit ball
of $\mathrm{Alg} \mathcal{L}$, so we can choose a contractive sequence $\{e_{\alpha}\}$ in $\mathrm{Alg} \mathcal{L}\cap \mathcal{K}(\mathcal{H})$, which is $\sigma$-weak-convergent to identity, such that $L_{a}e_{\alpha}=ae_{\alpha}=0$ for every $\alpha$. Thus, $a=0$.


It remains to verity that $\textrm{Im}\,\varphi=\mathcal{M}(\mathrm{Alg} \mathcal{L}\cap \mathcal{K}(\mathcal{H}))$. Consider an arbitrary element $(L, R)\in \mathcal{M}(\mathrm{Alg} \mathcal{L}\cap \mathcal{K}(\mathcal{H}))$. Then the sequences $\{L(e_{\alpha})\}$ and $\{R(e_{\alpha})\}$ are bounded by $\|L\|$ and $\|R\|$, in fact $\|L\|=\|R\|$, respectively. Since  $\mathrm{Alg} \mathcal{L}$ is a $\sigma$-weak closed subalgebra of $B(\mathcal{H})$ and the unit ball of $B(\mathcal{H})$  is $\sigma$-weak compact, $\{L(e_{\alpha})\}$ and $\{R(e_{\alpha})\}$ have points of accumulation $x_{L}\in \mathrm{Alg} \mathcal{L}$ and $x_{R}\in \mathrm{Alg} \mathcal{L}$. Passing, if necessary, to sub-sequences, we can suppose without loss of generality that
$$x_{L}=\lim\limits_{\alpha} L(e_{\alpha}), \ \ x_{R}=\lim\limits_{\alpha} R(e_{\alpha}),$$ in the $\sigma$-weak topology and $\|x_{L}\|\leq \|L\|$, $\|x_{R}\|\leq \|R\|$.

We claim that $x_{L}=x_{R}.$ Indeed, for any $a,b\in \mathrm{Alg} \mathcal{L}\cap \mathcal{F}_{1}(\mathcal{H})$, where $\mathcal{F}_{1}(\mathcal{H})$ is the set of all rank one operators in $B(\mathcal{H})$, Since  $\{e_{\alpha}\}$ converges to identity in the $\sigma$-weak topology, $\{e_{\alpha}^{*}\}$ converges to identity in the $\sigma$-weak topology. For any rank one operator $a$, $\{e_{\alpha}^{*}a^{*}\}$ converges to $a^{*}$ in norm, hence, $\{ae_{\alpha}\}$ converges to $a$ in norm. So, we have
\begin{equation}\label{aaa}
  ax_{L}b=a \lim\limits_{\alpha} L(e_{\alpha}) b=a \lim\limits_{\alpha} L(e_{\alpha}b)=aL(b)=R(a)b.
\end{equation}
\begin{equation}\label{bbb}
  ax_{R}b=a \lim\limits_{\alpha} R(e_{\alpha})b=\lim\limits_{\alpha} R(ae_{\alpha})b=R(a)b,
\end{equation}
 Hence, by  (\ref{aaa}) and (\ref{bbb}), if $a,b \in \mathrm{Alg} \mathcal{L}$ are rank one operators, we have
 \begin{equation}\label{777}
   ax_{L}b=ax_{R}b.
 \end{equation}

 Since $\mathrm{Alg} \mathcal{L}$ has the strong rank one property, in (\ref{777}), we can choose a sequence of rank one operators $\{r_{i}\}$ such that $\{r_{i}\}$ converges to identity in the strong operator topology. Hence, we have $$x_{L}=x_{R}.$$
 Let us denote $x:=x_{L}=x_{R}.$ Then
 \begin{equation}\label{555}
   aL(b)=axb=R(a)b
 \end{equation}
 for any $a,b\in \mathrm{Alg} \mathcal{L}\cap \mathcal{K}(\mathcal{H}).$

 In (\ref{555}), let $a=r_{i}$  and $\{r_{i}\}$ converges to identity in the strong operator topology, then $$L(b)=\lim\limits_{i} r_{i}L(b)=\lim\limits_{i} r_{i}xb=xb=L_{x}(b),$$ for any $b\in \mathrm{Alg} \mathcal{L}\cap \mathcal{K}(\mathcal{H})$

In (\ref{555}), let $b=r_{i}$ and $\{r_{i}\}$ converges to identity in the strong operator topology, then $$R(a)=\lim\limits_{i} R(a)r_{i}=\lim\limits_{i} axr_{i}=ax=R_{x}(a),$$ for any $a\in \mathrm{Alg} \mathcal{L}\cap \mathcal{K}(\mathcal{H}).$

The above shows that $\varphi(x)=(L, R).$ The proof is complete.\end{proof}

%
%

\begin{corollary}\label{CDCSL}
  Let $\mathcal{L}$ be a completely distributive CSL or a subspace lattice with two atoms on a complex separable Hilbert space $\mathcal{H}$. Then $\mathcal{M}(\mathrm{Alg} \mathcal{L}\cap \mathcal{K}(\mathcal{H}))$ and $\mathrm{Alg}\mathcal{L}$ are isomorphic.
\end{corollary}
%
%
%
%
%
%


Now, we recall another property: \textit{property $\mathbb{A}$}. An invertible element $u$ in a unital Banach algebra $\mathcal{A}$  is called \textit{doubly power-bounded} if
$\sup\limits_{k\in \mathbb{Z}}\|u^{k}\|<\infty.$
\begin{definition}
 A Banach algebra $\mathcal{A}$ has the \textit{property $\mathbb{A}$} if it is faithful and satisfies
  $$\mathcal{A}\subseteq \overline{\textrm{alg}\{u\in \mathcal{M}(\mathcal{A}):u \ \textrm{is doubly power-bounded} \}}^{s}.$$
  Here $\textrm{alg}S$ denotes the algebra generated by the set $S$, and $\overline{S}^{s}$ denotes the closure of $S$ with respect to the strict topology.
\end{definition}

It should be noted that if $I$ is the unit element of $\mathcal{M}(\mathcal{A})$, then $(I-2e)^{2}=I$ for any idempotent element  $e$ of $\mathcal{M}(\mathcal{A})$. Hence $I-2e$ is doubly power-bounded and $e=\frac{1}{2}I-\frac{1}{2}(I-2e)$.  It follows that every idempotent element belongs to  $\textrm{alg}\{u\in \mathcal{M}(\mathcal{A}):u \ \textrm{is doubly power-bounded}\}.$

If $\mathcal{A}$ is a Banach algebra having bounded approximate identity, then $\mathcal{A}$ has property $\mathbb{A}$ implies it is zero product determined due to Alaminos,  Bre$\check{s}$ar, Extremera and Villena (\cite[Theorem 5.20]{bresarbook}).

Now, we are in the position of verifying Theorem \ref{cdcslB}.

\begin{proof}[Proof of  Theorem \ref{cdcslB}]
  By Corollary \ref{CDCSL} and \cite[Corollary 1.15 ]{bresarbook},  Alg$\mathcal{L}$ is zero product
  determined if and only if  $\mathcal{M}(\textrm{Alg}\mathcal{L} \cap \mathcal{K}(\mathcal{H}))$ is. It remains  to show that $\mathcal{M}(\textrm{Alg} \mathcal{L}\cap \mathcal{K}(\mathcal{H}))$ is a zero product determined Banach algebra.

  By \cite[Lemmas 2.3, 2.10]{JK}, Alg$\mathcal{L}\cap \mathcal{K}(\mathcal{H})$ is the norm-closure of span of idempotents of Alg$\mathcal{N}$, and Alg$\mathcal{L}\cap \mathcal{K}(\mathcal{H})$ is dense in $\mathcal{M}(\textrm{Alg} \mathcal{L}\cap \mathcal{K}(\mathcal{H}))$ under the strict topology. Since the strict topology is weaker than the norm topology, $\mathcal{M}(\textrm{Alg} \mathcal{L}\cap \mathcal{K}(\mathcal{H}))$ is the strict-closure of span of idempotents.

  It follows that
 \begin{equation*}
    \mathcal{M}(\textrm{Alg}\mathcal{L}\cap \mathcal{K}(\mathcal{H}))\subseteq \overline{\textrm{alg}\{u\in \mathcal{M}(\textrm{Alg} \mathcal{L}\cap \mathcal{K}(\mathcal{H})): u \ \textrm{is doubly-power bounded}\}}^{\textit{s}}.
  \end{equation*}

  Hence $\mathcal{M}(\textrm{Alg} \mathcal{L}\cap \mathcal{K}(\mathcal{H}))$ has property $\mathbb{A}$. Therefore,  $\mathcal{M}(\textrm{Alg} \mathcal{L}\cap \mathcal{K}(\mathcal{H}))$ is zero product determined.
\end{proof}

The following corollary is a generalization of  Ghahramani's  results (\cite[Corollary2.8]{hogernestalgebra}).
\begin{corollary}
  Every nest algebra $\mathrm{Alg}\mathcal{N}$ on $\mathcal{H}$ is a zero product determined Banach algebra.
\end{corollary}

\begin{remark}
  In \cite{erdos}, Erdos gives an abstract characterization of nest algebras and partially answers which Banach algebras are isometrically isomorphic to some nest algebra. Thus, those abstract nest algebras described in \cite{erdos} are zero product determined.
\end{remark}

As an application, we study the local $n$-cocycles, homomorphisms on completely distributive commutative subspace lattice algebras  and subspace lattice algebras with two atoms through Theorem \ref{cdcslB}.

Let $\mathcal{X}$ and $\mathcal{Y}$ be Banach spaces. For $n\in  \mathbb{N}$, let $\mathcal{X}^{(n)}$ be the Cartesian product of $n$ copies of $\mathcal{X}$, and let $B^{n}(\mathcal{X}, \mathcal{Y})$ be the space of bounded $n$-linear mappings from $\mathcal{X}^{(n)}$ into $\mathcal{Y}$.

  Let $\mathcal{A}$ be a Banach algebra, and  $\mathcal{M}$  a Banach $\mathcal{A}$-bimodule. For
$n\in  \mathbb{N}$ and $T\in B^{n}(\mathcal{A},\mathcal{M})$, define
\begin{eqnarray*}
  \delta^{n}T : (a_{1},\ldots , a_{n+1}) &\mapsto&  a_{1}T(a_{2},\ldots, a_{n+1}) \\
   &+& \sum\limits^{n}_{j=1}(-1)^{j}T(a_{1},\ldots , a_{j-1}, a_{j}a_{j+1},\ldots, a_{n+1}) \\
  &+& (-1)^{n+1}T(a_{1},\ldots, a_{n})a_{n+1}.
\end{eqnarray*}
The elements of ker $\delta^{n}$ are called the \textit{ bounded n-cocycles}. $T$ is a \textit{bounded local n-cocycle} if, for each $\tilde{a} = (a_{1},\ldots, a_{n})\in \mathcal{A}^{(n)}$, there is a bounded  $n$-cocycle $T_{\tilde{a}}$ from $\mathcal{A}^{(n)}$ into $\mathcal{M}$ such that $T (\tilde{a}) = T_{\tilde{a}}(\tilde{a})$.
In \cite{Hou}, Hou and Fu show that every local 3-cocycle of a von Neumann algebra  into an arbitrary unital dual bimodule is a 3-cocycle.

\begin{proposition}\label{9898}
  Let $\mathcal{L}$ be a completely distributive commutative subspace lattice or a subspace lattice with two atoms on a separable Hilbert space $\mathcal{H}$. Then, for any Banach $\mathrm{Alg}\mathcal{L}$-bimodule $\mathcal{M}$ and $n\in \mathbb{N}$, every bounded local $n$-cocycle $T$ from $(\mathrm{Alg}\mathcal{L})^{(n)}$ into $\mathcal{M}$ is an $n$-cocycle.
\end{proposition}

\begin{proof}
  As stated in Theorem \ref{cdcslB}, Alg$\mathcal{L}$ is zero product determined.  Then the  result is obtained directly from \cite[Theorem 2.5]{ncocycles}.
\end{proof}
\begin{remark}
 In Proposition \ref{9898}, for $n=1$, it can be inferred that every bounded local derivation from Alg$\mathcal{L}$ into any Banach Alg$\mathcal{L}$-bimodule  $\mathcal{M}$ is a derivation. It is worth pointing out that the Alg$\mathcal{L}$-bimodules are not necessarily unital.
\end{remark}

A mapping $\varphi$ is called a \textit{weighted homomorphism} from an algebra $\mathcal{A}$ into an algebra $\mathcal{B}$, if there exists an invertible centralizer $W$ (i.e. $W(ab)=W(a)b=aW(b)$) on $\mathcal{B}$ and a homomorphism $\phi$ from $\mathcal{A}$ into $\mathcal{B}$ such that $$\varphi=W\phi.$$
By \cite[Theorem 7.4]{bresarbook} and Theorem \ref{cdcslB}, we obtain the following result.
\begin{proposition}\label{homo}
  Let $\mathcal{L}$ be a completely distributive commutative subspace lattice or a subspace lattice with two atoms on a
separable Hilbert space $\mathcal{H}$, and let $\mathcal{B}$ be a faithful Banach algebra satisfying $\mathcal{B}^{2}=\mathcal{B},$ where $\mathcal{B}^{2}$ denotes the  linear span of $\{b_{1}b_{2}: b_{1},b_{2}\in \mathcal{B}\}$. Then every surjective continuous zero product preserving linear mapping $\varphi : \mathrm{Alg} \mathcal{L}\rightarrow \mathcal{B}$ is a weighted homomorphism.
\end{proposition}
 \begin{remark}

Let $\mathcal{L}$ be as in Proposition \ref{homo}, if $\varphi$ is a surjective continuous homomorphism on $\textrm{Alg} \mathcal{L}$, then we obtain a complete description of surjective continuous homomorphisms on Alg$\mathcal{L}$.
\end{remark}


\section{Zero Lie product determined algerbas}\label{zLpd}

Zero Lie product determined Banach algebras are discussed in \cite{zlpd}. In this section we shall investigate  the zero Lie product determined property in triangular UHF algebras, locally matrix algebras etc. with the primary fact that if  a Banch algebra is weakly amenable and zero product determined with a bounded approximate identity, then it is zero Lie product determined.

A Banach algebra $\mathcal{A}$ is  \textit{weakly amenable} if every continuous derivation $\delta$ from $\mathcal{A}$ into its dual space $\mathcal{\mathcal{A}^{*}}$ is inner.

 A (Banach) algebra $\mathcal{A}$ is \textit{zero Lie product determined} if, for every (continuous) bilinear functional $\phi$ on $\mathcal{A}$ satisfying $\phi(x, y)=0$ whenever $x$ and
$y$ commute, there is a linear functional $\tau$ on $[\mathcal{A},\mathcal{A}]$ such that $\phi(x, y) = \tau([x, y])$ for all $x, y \in \mathcal{A}$, where $[\mathcal{A},\mathcal{A}]$ denotes the linear span of all commutators of the (Banach) algebra $\mathcal{A}$.

Let us emphasize that the definition here of a zero Lie product determined Banach algebra in which $\tau$ is not necessarily continuous has a slight difference  compared with Bre{\v s}ar's book \cite{bresarbook}.

\subsection{Triangular UHF algebras}

Let $\{p_{n}\}$ be a sequence of positive integers such that $p_{m}\rvert p_{n}$, whenever $m\leq n$. For each $n$, let $\mathcal{T}_{p_{n}}$ be the algebra of all $p_{n}\times p_{n}$ upper triangular complex matrices. For $m\leq n$, define $\sigma_{ p_{n},p_{m}}$ be the following mapping from $\mathcal{T}_{p_{m}}$ into $\mathcal{T}_{p_{n}}$, $$\sigma_{ p_{n},p_{m}}(x)=I_{d}\otimes x, d=p_{n}/p_{m},$$ for all $x\in \mathcal{T}_{p_{m}}.$

Then a \textit{triangular UHF algebra} $\mathcal{A}$ is a Banach algebra which is isometrically isomorphic to the following Banach algebra inductive limit: $$\mathcal{A}=\lim\limits_{\longrightarrow}(\mathcal{T}_{p_{n}}; \sigma_{ p_{n},p_{m}}).$$  Let $\varphi_{n}$ be the canonical map from $\mathcal{T}_{p_{n}}$ into $\mathcal{A}$ and set $\mathcal{T}_{p_{n}}^{\prime}= \varphi_{n}(\mathcal{T}_{p_{n}}).$ Then $\bigcup\limits_{n}\mathcal{T}_{p_{n}}^{\prime}$ is dense in $\mathcal{A}$.



\begin{theorem}\label{TUHFzlpd}
  Let $\mathcal{A}$ be a triangular UHF algebra, then $\mathcal{A}$ is  a zero Lie product determined Banach algebra.
\end{theorem}
\begin{proof}
   Let $\mathcal{B}=\cup^{\infty}_{n=1}\mathcal{T}_{p_{n}}$ and $\mathcal{A}=\overline{\mathcal{B}}^{\|\cdot\|}$, where $\mathcal{T}_{p_{n}}$ is the algebra of all $p_{n}\times p_{n}$ upper triangular complex matrices and satisfies $\mathcal{T}_{p_{i}}\subset \mathcal{T}_{p_{i+1}}.$

For every continuous bilinear functional $\varphi$ on $\mathcal{A}$ satisfying
\begin{eqnarray}\label{equ1}
[x, y]=0\Rightarrow \varphi(x, y)=0,
\end{eqnarray}
for $x, y\in \mathcal{A},$ we shall show that there exists a linear functional $\rho$ on $\mathcal{B}$ such that $\varphi(x, y)=\rho([x,y])$ for all $x,y\in \mathcal{B}.$

Indeed, for every $n\in \mathbb{N},$ $\mathcal{T}_{p_{n}}$ is a finite nest algebra.  Hence, by  \cite[Theorem 3.1]{nestweakly}, $\mathcal{T}_{p_{n}}$ is a zero Lie product determined Banach algebra. Let $\varphi_{n}$ be the restriction of  $\varphi$ on $\mathcal{T}_{p_{n}}$,  $\varphi_{n}$ satisfies the Equation (\ref{equ1}) on $\mathcal{T}_{p_{n}}$, thus there exists a  linear functional $\tau_{n}$ on $\mathcal{T}_{p_{n}}$ such that $\varphi_{n}(x, y)=\tau_{n}([x,y])$ for all $x,y\in \mathcal{T}_{p_{n}}.$

Since $\mathcal{T}_{p_{1}}$ can be embeded into $\mathcal{T}_{p_{2}}$ and $\tau_{2}\rvert_{\mathcal{T}_{p_{1}}}([x,y]) = \tau_{1}([x,y]),$ for any $x,y\in \mathcal{T}_{p_{1}}$, we finally construct a sequence of  linear functionals $\{\tau_{n}\} _{n \in \mathbb{N}}$ such that ${\tau_{1}} \subseteq {\tau_{2}} \subseteq \cdots \subseteq {\tau_{n}}\subseteq\cdots$. It follows that there exists a  linear functional $\rho$ on $\mathcal{B}$ such that $\rho\rvert_{\mathcal{T}_{p_{n}}}= {\tau_{n}}$.
%
For every $x,y\in \mathcal{B},$ there is a $\mathcal{T}_{p_{n}}$ for some $n$ such that $x,y\in \mathcal{T}_{p_{n}}$, and $\varphi(x,y)={\tau_{n}}([x,y])=\rho([x,y]).$



It remains to show that $\mathcal{A}$ is zero Lie product determined. Since $\mathcal{A}=\overline{\mathcal{B}}^{\|\cdot\|}$,  for any $a\in \mathcal{A}, b\in \mathcal{B}$, define $\widetilde{\rho}([a,b])=\lim\limits_{n\rightarrow\infty}\varphi(a_{n},b)=\varphi(a,b)$ where $\{a_{n}\}\subseteq \mathcal{B}, $ and $\{a_{n}\}$ converges to $a$ in the norm topology.

For any $a,c\in \mathcal{A}$,  there exists a sequence $\{c_{n}\}\subseteq \mathcal{B},$ and $\{c_{n}\}$ converges to $c$ in the norm topology. Since $\lim\limits_{n\rightarrow\infty}\widetilde{\rho}([a,c_{n}])=\lim\limits_{n\rightarrow\infty}
\varphi(a,c_{n})=\varphi(a,c),$  the limit of $\widetilde{\rho}([a,c_{n}])$ exists and  is denoted by $\widetilde{\rho}([a,c]).$ Hence, we define a functional on $\mathcal{A}$ and it is easy to check that $\widetilde{\rho}$  is well defined, linear and  $\widetilde{\rho} \rvert _{\mathcal{B}} = \rho$. This is what we desire.
\end{proof}



%

%

\subsection{Locally matrix algebras}

 An associative algebra $\mathcal{A}$ is a \textit{locally matrix algebra} if for each finite subset of $\mathcal{A}$ there exists a subalgebra $\mathcal{B}\subset\mathcal{A}$ containing this subset such that $\mathcal{B}\cong M_{n}(\mathbb{C})$ for some $n$.

\begin{theorem}\label{localmatrixalgebras}
  Suppose $\mathcal{A}$ is a  countable dimensional locally matrix algebra having a non-trivial idempotent $e$, then $\mathcal{A}$ is a zero Lie product determined algebra.
\end{theorem}
\begin{proof}
  Suppose at the first that $\mathcal{A}$ is unital, by the Köthe’s Theorem \cite{kothe}, we
can assume that algebra $\mathcal{A}\cong \otimes^{\infty}_{i=1}\mathcal{A}_{i}$, where $\mathcal{A}_{i}$ is finite dimensional matrix algebra and $\otimes^{\infty}_{i=1}\mathcal{A}_{i}$ denote the algebraic tensor product of $\{\mathcal{A}_{i}\}^{\infty}_{i=1}.$

  For any element $a_{1}\in \mathcal{A}_{1}$, we can represent $a_{1}$ by
$$a_{1}\otimes \mathbb{C}1 \otimes \mathbb{C}1 \cdots  \in \mathcal{A}_{1}\otimes \mathcal{A}_{2} \otimes \mathcal{A}_{3} \cdots.$$

It follows that $\mathcal{A}_{i}$ can be embeded  into $\mathcal{A}_{1}\otimes \mathcal{A}_{2}\otimes\cdots \otimes\mathcal{A}_{n},$ $n$ may be infinite. Hence, we have
$$\mathcal{A}=\bigcup\limits_{n\geq1}(\mathcal{A}_{1}\otimes \mathcal{A}_{2}\otimes\cdots \otimes \mathcal{A}_{n}).$$

 It is apparent that $\mathcal{A}_{i}\otimes \mathcal{A}_{j}\cong M_{n_{i}}(\mathbb{C})\otimes  M_{n_{j}}(\mathbb{C})\cong M_{n_{i}n_{j}\times n_{i}n_{j}}(\mathbb{C}).$ Hence, $\mathcal{A}_{i}\otimes \mathcal{A}_{j}$ is a zero Lie product determined algebra. By a similar argument of Theorem \ref{TUHFzlpd}, we conclude that $\mathcal{A}$ is a zero Lie product determined algebra.

Suppose now that $\mathcal{A}$ is not a  unital locally matrix algebra, there exists a sequence of idempotents $e_{1}=e, e_{2}, ...$ such that $$e_{i}\mathcal{A}e_{i}\subset e_{i+1}\mathcal{A}e_{i+1}, i\geq1, \ \mathrm{and} \ \bigcup\limits_{i\geq1}e_{i}\mathcal{A}e_{i} = \mathcal{A}.$$ Since every $e_{i}\mathcal{A}e_{i}$ is a  countable dimensional unital locally matrix algebra, it is zero Lie product determined. Thus, we can also prove that $\mathcal{A}$ is zero Lie product determined by a similar technique.
\end{proof}

%
%
%

\subsection{$\mathcal{J}$-subspace lattice algebras and algebras of measurable operators}

Followed by \cite{jslintro} in which the class of $\mathcal{J}$-subspace lattices was introduced, Longstaff et al. began to discuss this class in \cite{longstaff,oreste} and subsequently attracted by other researchers. It follows from these references that both pentagon subspace lattices and atomic Boolean subspace lattices are $\mathcal{J}$-subspace lattices. In the next main result, Theorem \ref{jsl}, we shall prove  that the algebra of all finite rank operators of a $\mathcal{J}$-subspace lattice algebra is actually zero Lie product determined.

 Given a subspace lattice $\mathcal{L}$ on a Hilbert space $\mathcal{H}$, put
  $$\mathcal{J}(\mathcal{L}) = \{K\in \mathcal{L} : K\neq \{0\}\ \textrm{and} \ K_{-}\neq\mathcal{H}\},$$ where $K_{-}= \bigvee \{L\in\mathcal{L} : L\supsetneq K\}$.

   A subspace lattice $\mathcal{L}$ on a Hilbert space $\mathcal{H}$ is said to be a  \textit{ $\mathcal{J}$-subspace lattice}, if

  \begin{itemize}
    \item[(1)] $\bigvee\{K : K\in \mathcal{J} (\mathcal{L})\} = \mathcal{H};$
    \item[(2)]  $\bigwedge\{K_{-}: K\in \mathcal{J} (\mathcal{L})\} = {0};$
    \item[(3)] $K\bigvee K_{-}= \mathcal{H}$ for every $K\in \mathcal{J} (\mathcal{L});$
    \item[(4)]  $K\bigwedge K_{-} = {0}$ for every $K\in \mathcal{J} (\mathcal{L})$.
  \end{itemize}

For $K\in \mathcal{J} (\mathcal{L})$, we write $\mathcal{F}_{1}(K) = \{x\otimes f : 0 \neq x\in K, 0\neq f\in K_{-}
^{\perp}$\}, where $K_{-}^{\perp}$ denotes the annihilator of
$K_{-}$. Let $\mathcal{F}(K)$ denote the linear span of $\mathcal{F}_{1}(K)$.

\begin{lemma}\label{fkzlpd}
  $\mathcal{F}(K)$ is  a zero Lie product determined algebra, where $K\in \mathcal{J} (\mathcal{L})$.
\end{lemma}
\begin{proof}
  By \cite[Lemma 3.6]{LULI}, $\mathcal{F}(K)$ is a locally matrix algebra. Due to Theorem \ref{localmatrixalgebras}, $\mathcal{F}(K)$ is  zero Lie product determined.
\end{proof}

\begin{theorem}\label{jsl}
  Suppose $\mathcal{L}$ is a $\mathcal{J}$-subspace lattice, then $\mathrm{Alg} \mathcal{L}\cap \mathcal{F}(\mathcal{H})$ is  a zero Lie product determined algebra.
\end{theorem}
\begin{proof}
   Apparently, $\mathrm{Alg} \mathcal{L}\cap \mathcal{F}(\mathcal{H})= \textrm{span}\{\mathcal{F}(K):K\in \mathcal{J}(\mathcal{L})\}.$ Thanks to \cite[Proposition 1.1(3)(4)]{jsllu}, for $K, L \in \mathcal{J} (\mathcal{L}), \mathcal{F}_{1}(K)\cap \mathcal{F}_{1}(L) = \emptyset $ if and only if $K\neq L$, and $\mathcal{F}_{1}(K)\mathcal{F}_{1}(L) = \{0\} $ if and only if $K\neq L.$ Hence, $\mathcal{F}(K)\cap \mathcal{F}(L) = \emptyset $ if and only if $K\neq L$, $\mathcal{F}(K)\mathcal{F}(L) = \{0\} $ if and only if $K\neq L,$ then we have
$$\mathrm{Alg} \mathcal{L}\cap \mathcal{F}(\mathcal{H})= {\textstyle\sum\limits_ { {\scriptscriptstyle K\in \mathcal{J}(\mathcal{L})}}}\oplus\mathcal{F}(K).$$ By Proposition \ref{fkzlpd}, $\mathrm{Alg} \mathcal{L}\cap \mathcal{F}(\mathcal{H})$ is zero Lie product determined.
\end{proof}

Let $\mathcal{M}$ be a von Neumann algebra, the
set $S(\mathcal{M})$ of all  measurable operators affiliated with $\mathcal{M}$ is a unital $*$-algebra when equipped with the algebraic operations of strong addition and multiplication and taking the adjoint of an operator. More details about this can be found in \cite{chilin}.

In the following proposition, we obtain that $S(\mathcal{M})$ is zero Lie product determined when $\mathcal{M}$ is a von Neumann algebra of Type $I_{n}.$

\begin{lemma}\cite[Corollary 3.11]{bresarbook}\label{abelianzlpd}
  If $\mathcal{B}$ is an abelian unital algebra, then $M_{n}(\mathcal{B})$ is a zero Lie product determined algebra.
\end{lemma}

\begin{proposition}
  If $\mathcal{M}$ is a von Neumann algebra of Type $I_{n}$, then $S(\mathcal{M})$ is a zero Lie product determined algebra.
\end{proposition}
\begin{proof}
  If $\mathcal{M}$ is Type $I_{n}$, then $\mathcal{M}\cong M_{n}(\mathcal{B})$ where  $\mathcal{B}$ is an abelian von Neumann algerba. In this case, $S(\mathcal{M})\cong M_{n}(S(\mathcal{B}))$, and $S(\mathcal{B})$ is an abelian algebra. By Proposition \ref{abelianzlpd}, $S(\mathcal{M})$ is  zero Lie product determined.\end{proof}
\subsection{Two-sided zero product determined  algebras}

We now introduce a definition of analytic version for an algebra to be two-sided zero product determined.

  A Banach algebra $\mathcal{A}$ is \textit{two-sided zero product determined} if every continuous bilinear functional $\varphi : \mathcal{A}\times \mathcal{A}\rightarrow \mathbb{C}$ with the property that for all $x, y\in \mathcal{A}$,
$$xy=yx=0\Rightarrow \varphi(x,y)=0,$$
  is of the form
$$\varphi(x,y)=\tau_{1}(xy)+\tau_{2}(yx)$$ for some continuous linear functionals $\tau_{1}, \tau_{2}$ on $\mathcal{A}$.

A proposition about two-zero product determined algebras appears immediately, which is a substantial generalization of \cite[Corollary 3.2]{twosided}.

\begin{proposition}
  Any finite nest algebra on a complex Hilbert space is two-zero product determined.
\end{proposition}
\begin{proof}
  First, any finite nest algebra is weakly amenable by \cite[Lemma 3.2]{nestweakly}, and it is zero product determined. Hence we can obtain it is two-zero product determined by \cite[Theorem 6.6]{bresarbook}.
  \end{proof}
%
%
%
%
%
%
%
%
%
%
%
%
%
%
%
%
%
%




\begin{thebibliography}{99}


\bibitem{zlpd}
Alaminos J, Bre\v{s}ar M,  Extremera J et al.
Zero Lie product determined Banach algebras, II.
J Math Anal Appl, 2019, 474: 1498--1511

\bibitem{orthalam}
Alaminos J, Extremera J,
  Villena A.
{Orthogonality preserving linear maps on group
  algebras}.
{Math Proc Cambridge Philos Soc,} 2015, 158: {493--504}


\bibitem{breasr2007}
Alaminos J, Extremera J,
  Villena A et al.
{Characterizing homomorphisms and derivations on
  {$C^*$}-algebras}.
{Proc Roy Soc Edinburgh Sect}, 2007, 137: {1--7}

\bibitem{commutators}
Alaminos J, Extremera J,
  Villena A et al.
{Commutators and square-zero elements in {B}anach
  algebras}.
{Q J Math}, 2016, 67: {1--13}

\bibitem{ABSL}
{Argyros S}, {Lambrou M},
  Longstaff W. {Atomic {Boolean} subspace
  lattices and applications to the theory of bases}.
  {Mem Amer Math Soc,}
  {1991}


\bibitem{twosided}
{Bajuk Z}, {Bre\v{s}ar M.}
{Two-sided zero product determined algebras}.
{Linear Algebra Appl,} 2022, 643: {125--136}

\bibitem{bresarbook}
Bre\v{s}ar M. {Zero product determined
  algebras}. Front Math {Cham: Springer},
  {2021}


\bibitem{bilinearbresar}
Bre\v{s}ar M, Semrl P.
{On bilinear maps on matrices with applications to
  commutativity preservers}.
{J Algebra}, 2006, 301: {803--837}

\bibitem{jsllu}
{Chen L}, Lu F.
{Local maps of {JSL} algebras}.
{Complex Anal Oper Theory}, 2019, 13: {1661--1674}

\bibitem{chen}
 Chen Q. {A study of reflexivity of operator algebras and several mappings on them.} PhD thesis, Tongji University, 2012

\bibitem{wupeiyuan}
Conway J, Wu P.
{The splitting of {$A(T_{1}\oplus T_{2})$} and related questions.}
 {Indiana Univ Math J,} 1977, 26: 41--56


\bibitem{erdos} Erdos J. {An abstract characterization of nest algebras}. Quart J Math Oxford (2), 1971, 22

\bibitem{hogernestalgebra}
Ghahramani H.
{Zero product determined some nest algebras}.
{Linear Algebra Appl,} 2013, 438: {303--314}


 \bibitem{nestweakly}
{Ghahramani H}, {Fallahi K},
  Khodakarami W.
{A note on zero Lie product determined nest algebras
  as Banach algebras}.
{Wavelet and Linear Algebra}, 2021, 8: {1--6}

\bibitem{Hou}
Hou, C, Fu, B.  Local 3-cocycles of von Neumann algebras.  Sci China Ser A, 2007, 50: 1240–1250

\bibitem{JK}
{Hadwin D}, Li J.
{Local derivations and local automorphisms}.
{J Math Anal Appl,} 2004, 290: {702--714}



 \bibitem{multiplierstopo}
Khan L, Mohammad N,
  Thaheem A.
{{Double multipliers on topological algebras}}.
{{Int J Math Math Sci,}} 1999, 22: {629--636}



 \bibitem{kothe}
K{\"o}the G.
{Schiefk{\"o}rper unendlichen {Ranges} {\"u}ber dem
  {Zentrum}}.
{Math Ann,} 1931, 105: {15--39}



 \bibitem{muitiplier} Larsen R.  {An introduction to the theory of multipliers}. Springer, Berlin, 1971



\bibitem{longstaff}
{Longstaff W}, {Nation J},
  Panaia O.
{Abstract reflexive sublattices and completely
  distributive collapsibility}.
{Bull Austral Math Soc,} 1998, 58:  {245--260}

\bibitem{oreste}
{Longstaff W}, Panaia O.
\emph{{${J}$}-subspace lattices and subspace
  {M}-bases}.
{Studia Math,} 2000, 139: {197--212}


\bibitem{LULI}
{Lu F}, Li P.
{Algebraic isomorphisms and {J}ordan derivations of
  {${J}$}-subspace lattice algebras}.
{Studia Math,} 2003, 158:  {287--301}

\bibitem{marcouxabelian}
Marcoux L, Popov A.
{Abelian, amenable operator algebras are similar to {$C^*$}-algebras.} {Duke Math J,}  2016, 165: 2391--2406

\bibitem{MT}
Manuilov V, Troitsky E.
 {{Hilbert \(C^*\)-modules. Translated from the Russian by the
  authors}}. {Providence, RI:
  American Mathematical Society (AMS)}, {2005}

\bibitem{chilin}
Muratov M, Chilin V.
{Topological algebras of measurable and locally
  measurable operators}.
{J Math Sci New York,} 2019, 239: {654--705}

\bibitem{murphy}
{Murphy G.} {$C^{*}$-algebras and
  operator theory}. {Boston MA etc.: Academic Press,}
  {1990}

\bibitem{jslintro}
Panaia O. {Quasi-spatiality of isomorphisms
  for certain classes of operator algebras}. Diss University of Western Australia, 1995

\bibitem{ncocycles2008}
{Samei E.}
{Local properties of the {Hochschild} cohomology of $C^{*}$-algebras}.
{J Aust Math Soc,} 2008, 84: {117--130}

\bibitem{ncocycles}
Samei E.
{Reflexivity and hyperreflexivity of bounded {{\(n\)}}-cocycles from group algebras}.
{Proc Am Math Soc,} 2011, 139:  {163--176}

%
%
%
%
%
\end{thebibliography}
\end{document}